\documentclass[a4paper,reqno]{amsart}
\usepackage{amsmath, amssymb, amsthm}

\usepackage[mathscr]{eucal}

\newtheorem{thm}{\textbf{Theorem}}[section]
\newtheorem{prop}[thm]{\textbf{Proposition}}
\newtheorem{lem}[thm]{\textbf{Lemma}}

\theoremstyle{definition}
\newtheorem{defn}[thm]{\textbf{Definition}}
\newtheorem{rem}[thm]{{\textbf Remark}}

\makeatletter
    
    \@addtoreset{equation}{section}
  \makeatother

\newcommand{\olapla}{\overline\Delta}
\newcommand{\onabla}{\overline\nabla}

\newcommand{\lapla}{\Delta}

\newcommand{\p}{\phi}

\title{Biharmonic Lagrangian submanifolds \\ in K{\em ${\ddot {\bf A}}$}hler manifolds}

\author{Shun Maeta${}^{\ast}$}

\author{Hajime Urakawa${}^{\dagger}$} 

\thanks{
${}^{\ast}$
supported by 
Research Fellowships of the Japan Society for the Promotion of Science for Young Scientists, No.~23-6949.\\
\qquad
${}^{\dagger}$
supported by the Grants-in-Aid for Scientific Research (C), 
Japan Society for the Promotion of Science, No.21540207.}
\thanks{2010~{\em Mathematics Subject Classification.}~primary 58E20, secondary 53C43.}

\begin{document}

\maketitle
\markboth{Biharmonic Lagrangian submanifolds} 
{Shun Maeta and Hajime Urakawa}



\begin{abstract}
We give necessary and sufficient conditions for a Lagrangian submanifold of a K\"ahler manifold to be biharmonic (cf. Theorem \ref{main1}).
 Furthermore, we classify biharmonic PNMC Lagrangian submanifolds in the complex space forms (cf. Theorem \ref{blhsc}).
 \end{abstract}



\vspace{10pt}
\begin{flushleft}
{\large {\bf Introduction}}
\end{flushleft}

 Theory of harmonic maps has been applied into various fields in differential geometry.
 The harmonic maps between two Riemannian manifolds are
 critical maps of the energy functional $E(\p)=\frac{1}{2}\int_M\|d\p\|^2v_g$, for smooth maps $\p:M\rightarrow N$,
 whose Euler-Lagrange equation is the {\em tension field} $\tau(\p)$ of $\p$ vanishes.
 
On the other hand, in 1983, J. Eells and L. Lemaire \cite{jell1} proposed the problem to consider the {\em polyharmonic maps of order $k$}:
 they are critical maps of the functional 
 \begin{align*}
 E_{k}(\p)=\int_Me_k(\p)v_g,\ \ (k=1,2,\dotsm),
 \end{align*}
 where $e_k(\p)=\frac{1}{2}\|(d+\delta)^k\p\|^2$ for smooth maps $\p:M\rightarrow N$, where $\delta$ is the codifferentiation.
G.Y. Jiang \cite{jg1} studied the first and second variational formulas of the bi-energy $E_2$ $(k=2)$ which is written as
\begin{equation}
E_2(\p)=\int _M|\tau(\p)|^2 v_g,
\end{equation}
and critical maps of $E_2$ are called {\em biharmonic maps}. There have been extensive studies on biharmonic maps.
 Harmonic maps are always biharmonic maps by definition. Thus, one of our center problem is to find
 non-harmonic biharmonic maps.
 
Recently, T. Sasahara classified \cite{ts1} 2-dimensional biharmonic Lagrangian submanifolds in 
the $2$-dimensional complex space forms.

In this paper, we first show the biharmonic equations for a Lagrangian submanifold $M^m$ of  a K\"ahler manifold $(N^m, J, \langle \cdot,\cdot \rangle)$
 of complex $m$ dimension (cf. Theorem $\ref{main1}$).
Next, we give necessary and sufficient conditions for biharmonic Lagrangian submanifolds in complex space forms (cf. Proposition \ref{n-sbihcsf}).
Finally, we classify biharmonic Lagrangian $H$-umbilical submanifolds of the complex space form $(N^m(4\epsilon), J,\langle \cdot,\cdot \rangle)$
 which has parallel normalized mean curvature vector field (say briefly, PNMC) (cf. \cite{absmco1},~\cite{byc6})
(cf. Definition \ref{PNMC},~Theorem $\ref{blhsc}$).

 In $\S \ref{preliminaries}$, we introduce notation and fundamental formulas of the tension field.
In $\S \ref{n-s}$,  we consider biharmonic Lagrangian submanifolds in K\"ahler manifolds and  give necessary and sufficient conditions.
In the complex space form, we give necessary and sufficient conditions in $\S \ref{CSF}$.
In $\S \ref{H-umbilical}$, we classify all the  biharmonic PNMC Lagrangian $H$-umbilical submanifolds in complex space forms.
\vspace{10pt}

\noindent 
{\bf Acknowledgements.} 
We would like to thank Professor Kazuo Akutagawa who gave to us many helpful advice and Professor Toru Sasahara who gave crucial comments to the first draft.
 The first author also would like to thank Professor Ye-Lin Ou for useful comments.

\vspace{30pt}
\section{Preliminaries}\label{preliminaries}
In this section, we give the necessary notation on biharmonic maps for later use.

Let $(M,g)$ be an $m$ dimensional compact Riemannian manifold,
 $(N,\langle \cdot,\cdot \rangle)$ an $n$ dimensional one,
 and $\p:M\rightarrow N$, a smooth map.
 We use the following notation.
 The second fundamental form of $\p$ is a covariant differentiation $\widetilde\nabla d\p$ of $1$-form $d\p$,
 which is a section of $\odot ^2T^*M\otimes \p^{-1}TN$.
For every vector fields $X,Y\in \frak{X} (M)$ on $M$, 
 \begin{equation*}
 \begin{split}
(\widetilde\nabla d\p)(X,Y)
=&(\widetilde\nabla_X d\p)(Y)=\overline\nabla_Xd\p(Y)-d\p(\nabla_X Y)\\
=&\nabla^N_{d\p(X)}d\p(Y)-d\p(\nabla_XY). 
 \end{split}
 \end{equation*}
 Here, $\nabla, \nabla^N, \overline \nabla, \widetilde \nabla$ are the connections on the bundles $TM$,
 $TN$, $\p^{-1}TN$ and $T^*M\otimes \p^{-1}TN$, respectively.
  
 If $M$ is compact,
 we consider critical maps of the energy functional
 \begin{align*}
 E(\p)=\int_M e(\p) v_g,
 \end{align*}
where $e(\p)=\frac{1}{2}\|d\p\|^2=\frac{1}{2}\sum^m_{i=1}\langle d\p(e_i),d\p(e_i)\rangle$
 is the {\em energy density} of $\p$, $\{e_i\}_{i=1}^m$ is a locally defined orthonormal frame field on $(M,g)$,
  and the inner product 
 $\langle \cdot ,\cdot \rangle$ is a Riemannian metric of $N$. 
 We also denote that $\langle \cdot, \cdot \rangle $ is an induced metric $\p^*\langle \cdot, \cdot \rangle$.
 The {\em tension \ field} $\tau(\p)$ of $\p$ is defined by
 \begin{align*}
 \tau(\p)=\sum^{m}_{i=1}(\widetilde \nabla d\p)(e_i,e_i)=\sum^m_{i=1}(\widetilde \nabla _{e_i}d\p)(e_i).
 \end{align*}
 Then, $\p$ is a {\em harmonic map} if and only if $\tau(\p)=0$.
 
 As for the bi-energy $E_2$, G. Y. Jiang \cite{jg1} showed the first and second variational formula.
  $\p$ is called {\em biharmonic maps} if {\em bitension field} $\tau_2(\p)$ vanishes, that is, 
  $$\tau_2(\p)=\olapla \tau(\p)-\sum^n_{i=1}(\tau(\p),d\p(e_i))d\p(e_i)=0,$$
  where $R^N$ is the curvature tensor field
\begin{align*}
R^N(U,V)
=\nabla^N_U \nabla^N_V - \nabla^N _V \nabla^N_U-\nabla^N_{[U,V]},
\ \ \ \ (U,V\in \frak{X} (N)),
\end{align*}
 $\olapla
=\onabla^* \onabla
=-\sum^m_{k=1}(\onabla_{e_k}\onabla_{e_k}
-\onabla_{\nabla_{e_k}e_k})$ 
the {\em rough Laplacian},
 and $\onabla$, the induced connection on the induced bundle $\p^{-1}TN$.

The Gauss formula is given by
\begin{equation}
\nabla^N_XY=d\p(\nabla _XY)+B(X,Y),\ \ \ \ X,Y\in \frak{X}(M),
\end{equation}
where $\nabla^N$, $\nabla$ is the Levi-Civita connection on $N$ and $M$ respectively, and $B$, the second fundamental form.
The Weingarten formula is given by
\begin{equation}\label{Wformula}
\nabla^N_X \xi =-A_{\xi}X+\nabla^{\perp}_X{\xi},\ \ \ X\in \frak{X}(M),\  \xi \in \Gamma(TM^{\perp}),  
\end{equation}
where $A_{\xi}$ is the shape operator for a normal vector field $\xi$ on $M$ and $\nabla^{\perp}$ stands for the normal connection of the normal bundle on $M$ in $N$.
It is well known that $B$ and $A$ are related by
\begin{equation}\label{C formula}
\langle B(X,Y), \xi \rangle=\langle A_{\xi}X,Y \rangle.
\end{equation}
The curvature tensor $R^N$ on $N$ satisfies Gauss Codazzi equation:
\begin{equation}\label{GCeq}
\begin{aligned}
R^N(X,Y)Z
&=R(X,Y)Z+A_{B(X,Z)}Y-A_{B(Y,Z)}X\\
&\ \ \ +(\nabla_X^{\perp}B)(Y,Z)-(\nabla^{\perp}_Y B)(X,Z),
\end{aligned}
\end{equation}
where
 $\nabla^{\perp}B$ is given by 
$$(\nabla^{\perp}_X B)(Y,Z)=\nabla^{\perp}_X(B(Y,Z))-B(\nabla_XY,Z)-B(Y,\nabla_XZ).$$
If $\p:(M,g)\rightarrow (N,h)$ is a biharmonic isometric immersion, then $M$ is called a {\em biharmonic submanifold}.
In this case, the tension field is obtained as follows:
$\tau(\p)=m\,{\bf H},$
where ${\bf H}$ is the harmonic mean curvature vector along $\p$.
Thus, the bitension field $\tau_2(\p)$ given as:
\begin{equation}
\tau_2(\p)=m\left\{\olapla\,{\bf H}-\sum^m_{i=1}R^N({\bf H},d\p(e_i))d\p(e_i)\right\}.
\end{equation}
Therefore, $\p$ is biharmonic if and only if
\begin{equation}
\olapla\,{\bf H}-\sum^m_{i=1}R^N({\bf H},d\p(e_i))d\p(e_i)=0.
\end{equation}

\qquad\\
\qquad\\

\section{Necessary and sufficient conditions for biharmonic Lagrangian submanifolds in K\"ahler manifolds}\label{n-s}
In this section, we give necessary and sufficient conditions for a biharmonic Lagrangian submanifold in a K\"ahler manifold. 

Let us recall a fundamental material on a Lagrangian submanifold of a K\"ahler manifold following Chen and Ogiue \cite{bycko1}.

Let $(N^m,J,\langle \cdot , \cdot \rangle)$ be a K\"ahler manifold, where $J$ is the almost complex structure and complex dimension $m$ and $\langle \cdot , \cdot \rangle$ the K\"ahler metric,
 namely 
$\langle JU,JV \rangle=\langle U,V \rangle$ and $d\Phi=0$,
 where $\Phi(U,V)=\langle U,JV\rangle,\ (U, V\in \frak{X}(N))$ is the fundamental 2-form.
 Let $(M^m,g)$ be a Lagrangian submanifold of a K\"ahler manifold $(N^m,J,\langle \cdot , \cdot \rangle)$,
  that is, for all $x \in M$, $J(T_x(M))\subset T_xM^{\perp}_x$, where we also denote that $J$ is the almost complex structure on $M$, $T_x(M)$ denotes the tangent space of $M$
  at $x$ and $T_xM^{\perp}$ the normal space at $x$. 
 Then, it is well known that the following equations hold:
 \begin{align}
 \nabla^{\perp}_X JY=J(\nabla_X Y),
 \end{align}
\begin{equation}\label{R^N1}
R^N(JX,JY)=R^N(X,Y),
\end{equation}
for vector fields $X,~Y \in \frak{X}(M)$, and 
\begin{equation}\label{R^N2}
R^N(U,V)\cdot J= J\cdot R^N(U,V),
\end{equation}
for vector fileds $U,~V \in \frak{X}(N)$.

To show the biharmonic equations for a Lagrangian submanifold of a K\"ahler manifold, we need the following lemma.

\vspace{5pt}

\begin{lem}\label{bihlag}
Let $\p:(M,g)\rightarrow(N,\langle \cdot , \cdot \rangle)$ be an isometric immersion between Riemannian manifolds $(M,g)$ and $(N,\langle \cdot , \cdot \rangle)$. Then it is biharmonic if and only if 
\begin{align}
{\rm trace}_g\left(\nabla A_{\bf H}\right)+{\rm trace}_g\left(A_{\nabla^{\perp}_{\cdot}{\bf H}}(\cdot)\right)-\left(\sum_{i=1}^mR^N({\bf H},e_i)e_i\right)^{T}=0,\\
\lapla^{\perp}{\bf H}+{\rm trace}_g B(A_{\bf H}(\cdot),\cdot)-\left(\sum_{i=1}^mR^N({\bf H},e_i)e_i\right)^{\perp}=0,
\end{align}
where $(\cdot)^T$ is the tangential part and $(\cdot)^{\perp}$ the normal part.
\end{lem}

\vspace{5pt}

\begin{proof}
Due to $(\ref{Wformula})$, we have
\begin{align*}
\onabla_X{\bf H}
&=-A_{\bf H}(X)+\nabla^{\perp}_X{\bf H},
\end{align*}
\begin{align*}
\onabla_Y\onabla_X {\bf H}
&=-\nabla_YA_{\bf H}(X)-B(Y,A_{\bf H}(X))
+\nabla^{\perp}_Y\nabla^{\perp}_X{\bf H} -A_{\nabla^{\perp}_X {\bf H}}Y,
\end{align*}
for all vector field, $X,\,Y$ on $M$.

Thus, we have
\begin{align*}
\olapla\,{\bf H}
=&-\sum^m_{i=1}\left\{-\nabla_{e_i}A_{\bf H}(e_i)-B(e_i,A_{\bf H}(e_i))
+\nabla^{\perp}_{e_i}\nabla^{\perp}_{e_i}{\bf H}\right.\\
&\left.\hspace{50pt}-A_{\nabla^{\perp}_{e_i}{\bf H}}(e_i)
+A_{\bf H}(\nabla_{e_i}e_i)-\nabla^{\perp}_{\nabla_{e_i}e_i}{\bf H} \right\} .
\end{align*}
Dividing this into the tangential and normal part, we obtain Lemma $\ref{bihlag}$.
\end{proof}

 Then we obtain the following theorem.
\vspace{5pt}

\begin{thm}\label{main1}
Let $(N^m,J,\langle \cdot , \cdot \rangle)$ be a K\"ahler manifold of complex dimension $m$,
$(M^m,g)$ a Lagrangian submanifold in $(N^m,J,\langle \cdot , \cdot \rangle)$.
Then for an isometric immersion $\p$ from $(M^m,g)$ into  $(N^m,J,\langle \cdot , \cdot \rangle)$, it is biharmonic if and only if
 the following two equations hold:
\begin{equation}
\begin{aligned}
&{\rm trace}_g\left(\nabla A_{\bf H}\right)+{\rm trace}_g\left(A_{\nabla^{\perp}_{\cdot}{\bf H}}(\cdot)\right)\\
&-\sum^m_{i=1}\left \langle {\rm trace}_g
\left(\nabla^{\perp}_{e_i}B\right)-{\rm trace}_g\left(\nabla^{\perp}_{\cdot}B\right)(e_i,\cdot),{\bf H}\right\rangle e_i
=0,
\end{aligned}
\end{equation}

\begin{equation}
\begin{aligned}
&\lapla^{\perp}{\bf H}+{\rm trace}_g B\left(A_{\bf H}(\cdot),\cdot\right)
+\sum_{i=1}^m Ric^N(J{\bf H},e_i)Je_i\\
&-\sum_{i=1}^m Ric(J{\bf H},e_i)Je_i
-J~{\rm trace}_g A_{B(J{\bf H},\cdot)}(\cdot)+mJ A_{\bf H}(J {\bf H})
=0,
\end{aligned}
\end{equation}
where $Ric$ and $Ric^N$ are the Ricci tensor of $(M^m,g)$ for the Riemannian metric $g$ on $M^m$
induced from $\p$ and $(N^m,\langle \cdot , \cdot \rangle)$, respectively. 
 Here, ${\rm trace}_g\left(A_{\nabla_{\bullet}^{\perp}{\bf H}}(\bullet)\right)$ stand for 
 $\displaystyle \sum^m_{i=1}A_{\nabla_{e_i}^{\perp}{\bf H}}(e_i)$, and so on.
\end{thm}

\begin{proof}
The harmonic mean curvature vector ${\bf H}$ can be described as ${\bf H}=J Z$ for a vector field $Z$ on $M$.
 Using $(\ref{R^N1})$ and $(\ref{R^N2})$,  we obtain
\begin{equation*}
\sum^m_{i=1}\left\langle R^N({\bf H},e_i)e_i, JX\right\rangle
=\sum^m_{i=1}\left\langle R^N(Z, Je_i)Je_i,X\right\rangle,
 \end{equation*}
which implies that
\begin{align*}
\sum_{i=1}^m\left\langle R^N({\bf H},e_i)e_i,JX\right\rangle
+\sum_{i=1}^m \left\langle R^N(Z,e_i)e_i,X\right\rangle 
=Ric^N(Z,X).
\end{align*}
By Gauss Codazzi equation $(\ref{GCeq})$, we have
\begin{align}\label{5-5}
&\sum_{i=1}^m\left\langle R^N({\bf H},e_i)e_i,JX\right\rangle \notag\\
=&Ric^N(Z,X)\notag\\
&-\left\{\sum_{i=1}^m \left\langle R(Z,e_i)e_i,X\right\rangle
+\sum_{i=1}^m \left\langle B(Z,e_i),B(e_i,X)\right\rangle
-\sum^m_{i=1}\left\langle B(e_i,e_i),B(Z,X)\right\rangle\right\} \\
=&Ric^N(Z,X)-Ric(Z,X)
-\sum_{i=1}^m\langle B(Z,e_i),B(e_i,X)\rangle
+m\langle{\bf H},B(J{\bf H},X)\rangle\notag\\
=&-Ric^N(J{\bf H},X)+Ric(J{\bf H},X)\notag\\
&+\sum_{i=1}^m\langle B(J{\bf H},e_i),B(e_i,X)\rangle
-m\langle{\bf H},B(J{\bf H}, X)\rangle.\notag
\end{align}
From $(\ref{5-5})$, we have
\begin{align}\label{5-6}
&\left(\sum_{i=1}^m R^N({\bf H},e_i)e_i\right)^{\perp}\notag\\
=&\sum_{j=1}^m\left\langle \sum_{i=1}^m R^N({\bf H}, e_i)e_i,Je_j\right\rangle Je_j\\
=&-\sum_{j=1}^mRic^N(J{\bf H},e_j)Je_j
+\sum_{j=1}^mRic (J{\bf H},e_j)Je_j\notag\\
&+\sum_{i,j=1}^m \langle B(J{\bf H},e_i),B(e_i,e_j)\rangle Je_j
-\sum_{j=1}^m m\langle{\bf H},B(J{\bf H},e_j)\rangle J e_j.\notag
\end{align}

By $(\ref{C formula})$, 
\begin{align}\label{6-7}
\sum^m_{i,j=1}\langle B(J{\bf H},e_i),B(e_i,e_j)\rangle Je_j
=&\sum^m_{i,j=1}\left\langle A_{B(J{\bf H},e_i)}(e_i),e_j\right\rangle Je_j\notag\\
=&\sum^m_{i,j=1}\left\langle J A_{B(J{\bf H},e_i)}(e_i),Je_j\right\rangle Je_j\\
=&\sum^m_{i=1}J A_{B(J {\bf H},e_i)}(e_i)\notag\\
=&~J~ {\rm trace}_g A_{B(J{\bf H},\cdot)}(\cdot),\notag
\end{align}
 and 
 
\begin{align}\label{6-8}
\sum^m_{j=1}\langle{\bf H},B(J{\bf H},e_j)\rangle Je_j
=&\sum^m_{j=1}\langle A_{\bf H}(J{\bf H}),e_j\rangle Je_j\notag\\
=&\sum^m_{j=1}\langle JA_{\bf H}(J{\bf H}),Je_j\rangle Je_j\\
=&~J A_{\bf H}(J{\bf H}).\notag
\end{align}

Combining $(\ref{5-6})$, $(\ref{6-7})$ and $(\ref{6-8})$, we obtain
\begin{align}\label{6-9}
\left(\sum_{i=1}^m R^N({\bf H},e_i)e_i\right)^{\perp}
=&-\sum_{j=1}^mRic^N(J{\bf H},e_j)Je_j
+\sum_{j=1}^mRic (J{\bf H},e_j)Je_j\\
&+J~ {\rm trace}_g A_{B(J{\bf H},\cdot)}(\cdot)
-mJ A_{\bf H}(J{\bf H})\notag.
\end{align}

By $(\ref{GCeq})$,
\begin{align}
\left(\sum^m_{i=1}R^N({\bf H},e_i)e_i\right)^{T}
=&\sum^m_{i,j=1}\left\langle R^N({\bf H},e_i)e_i,e_j\right\rangle e_j\\
=&\sum^m_{i,j=1}\left\langle \left(\nabla^{\perp}_{e_j}B\right)(e_i,e_i)
-\left(\nabla^{\perp}_{e_i}B\right)(e_j,e_i),{\bf H}\right\rangle e_j.\notag
\end{align}
Applying Lemma $\ref{bihlag}$, we obtain the theorem.
\end{proof}

\qquad\\
\qquad\\

\section{Biharmonic Lagrangian submanifolds in complex space forms}\label{CSF}
In this section, we give necessary and sufficient conditions to be biharmonic for Lagrangian submanifolds in K\"ahler manifolds.
 
 Let $N=N^m(4\epsilon)$ be the complex space form of complex dimension $m$ and constant holomorphic sectional curvature $4\epsilon$.
The curvature tensor $R^N$ is given by
\begin{align*}
&R^N(U,V)W\\
&=\epsilon\{\langle V,W \rangle U-\langle U,W \rangle V+\langle W,JV\rangle JU
-\langle W,JU \rangle JV+2\langle U,JV \rangle JW\},\notag
\end{align*}
for vector fields $U,~V,~W \in \frak{X}(N)$, where $\langle \cdot, \cdot \rangle$ is the Riemannian metric on $N^m(4\epsilon)$ and $J$, the almost complex structure of $N^m(4\epsilon)$.

Every simply connected complex space form is holomorphically isometric to either the complex projective space $\mathbb{CP}^m(4\epsilon)$,
 the complex Euclidean space $\mathbb{C}^m$ or the complex hyperbolic space $\mathbb{CH}^m(4\epsilon)$ according to $\epsilon >0, \epsilon =0, \epsilon <0$.

\vspace{5pt}

Using Lemma $\ref{bihlag}$, we obtain the following proposition which will be used in the next section.
\begin{prop}\label{n-sbihcsf}
Let $(N^m(4\epsilon),J,\langle \cdot , \cdot \rangle)$ be a complex space form of complex dimension $m$,
$(M^m,g)$ a Lagrangian submanifold in $(N^m(4\epsilon),J,\langle \cdot , \cdot \rangle)$.
Then for an isometric immersion $\p$ from $M^m$ into $(N^n(4\epsilon),J,\langle \cdot , \cdot \rangle)$, it is biharmonic if and only if
\begin{align}
{\rm trace}_g\left(\nabla A_{\bf H}\right)+{\rm trace}_g\left(A_{\nabla^{\perp}_{\cdot}}(\cdot)\right)=0,\label{tan}\\
\lapla^{\perp}{\bf H}+{\rm trace}_g B\left(A_{\bf H}(\cdot),\cdot\right)-(m+3)\epsilon {\bf H}=0.\label{nor}
\end{align}
\end{prop}

\vspace{5pt}

\begin{proof}
\begin{align*}
\sum^m_{i=1}R^N({\bf H},d\p(e_i))d\p(e_i)
=&\epsilon \sum^m_{i=1} \left\{\langle d\p(e_i),d\p(e_i) \rangle {\bf H}-\langle d\p(e_i),{\bf H}\rangle d\p(e_i)\right.\\
&\hspace{25pt}+\langle d\p(e_i),J d\p(e_i)\rangle J{\bf H}
-\langle d\p(e_i),J{\bf H} \rangle J d\p(e_i)\notag\\
&\hspace{25pt}\left.+2\langle {\bf H},J d\p(e_i)\rangle J \p(e_i)\right\}\notag\\
=&\left\{m{\bf H}+\sum^m_{i=1} \langle {\bf H}, Jd\p(e_i)\rangle Jd\p(e_i)
+2{\bf H}\right\}\notag\\
=&(m+3)\epsilon {\bf H}.\notag
\end{align*}
Using this and Lemma $\ref{bihlag}$, we have the proposition.
\end{proof}

\qquad\\
\qquad\\

\section{Biharmonic Lagrangian $H$-umbilical submanifolds in complex space forms}\label{H-umbilical}
In this section, we classify biharmonic PNMC Lagrangian $H$-umbilical submanifolds in complex space forms.
 First, we recall several notions.
 
B. Y. Chen introduced the notion of Lagrangian $H$-umbilical submanifolds \cite{byc4}:

\begin{defn}[\cite{byc4}]
If a Lagrangian submanifold in a K\"ahler manifold has the second fundamental form takes the following:
\begin{equation}\label{4.1}
\left\{
\begin{aligned}
&B(e_1,e_1)=\lambda Je_1,\ \  B(e_i,e_i)=\mu Je_1,\\
&B(e_1,e_i)=\mu Je_i,\ \   B(e_i,e_j)=0, \ (i\neq j),\ \ \ i,j=2,\cdots ,m,
\end{aligned}
\right.
\end{equation}
for suitable function $\lambda, \mu$ with respect to some a suitable orthonormal frame field $\{e_1,\cdots, e_m\}$ on $M$, then it is called to be a {\em Lagrangian $H$-umbilical submanifold}.
\end{defn}
Lagrangian $H$-umbilical submanifolds are the simplest Lagrangian submanifolds next to totally geodesic ones.
Because it is known that there are no totally umbilical Lagrangian submanifolds in a complex space form $N^m(4\epsilon)$ with $m\geq2$,
 we should consider $H$-umbilical Lagrangian submanifolds.
 
In this case,  the harmonic mean curvature vector ${\bf H}$ can be denoted by
$${\bf H}=\frac{\lambda+(m-1)\mu}{m}Je_1.$$
 Hereinafter, we put $a=\frac{\lambda+(m-1)\mu}{m}.$
 \begin{rem}
 The class of Lagrangian $H$-umbilical submanifolds of the complex space forms includes the following interesting submanifolds:\\
 (1) the Whitney's sphere in the complex Euclidean space (\cite{byc3}), \\
 (2) twistor holomorphic Lagrangian surfaces in the complex projective plane(\cite{icfu1},~\cite{byc4}).\\
Furthermore, Lagrangian $H$-umbilical submanifolds in complex space forms
 were classified (\cite{byc3},~\cite{byc4},~\cite{byc5}).
 \end{rem}
 
 \vspace{5pt}

B.Y. Chen also introduced PNMC submanifolds (cf. \cite{absmco1},~\cite{byc6}):
\begin{defn}[\cite{absmco1}, \cite{byc6}]\label{PNMC}
A submanifold $M$ in a Riemannian manifold is said to have {\em parallel normalized
mean curvature vector field} (PNMC) if it has nowhere zero mean curvature and the
unit vector field in the direction of the mean curvature vector field is parallel in the
normal bundle, i.e.
\begin{equation}
\nabla^{\perp} \left(\frac{\bf H}{|\bf H|}\right) = 0.
\end{equation}
\end{defn}

We recall the Legendrian immersion and its fundamental formulas (cf. \cite{ts3},~\cite{hr1}):

 An isometric immersion $\tilde{f}:M^m\rightarrow S^{2m+1}$ is called {\em Legendrean}
 if $\sqrt{-1}\tilde{f}\perp\tilde{f}_{\*}(TM)$ and  $\langle \Phi \tilde{f}_*(TM), \tilde{f}_*(TM)\rangle =0,$ where, $\Phi$ denote the projection of the complex structure $J$ on the tangent bundle of $S^{2m+1}$, and
  $\langle \cdot,\cdot\rangle$ the inner product on $\mathbb{C}^{m+1}$ induced from $g_0$.

Let $f:M^m\rightarrow \mathbb{CP}^m$ be a Lagrangian immersion. Then, there exists a Legendrian immersion $\tilde{f}: M^m\rightarrow S^{2m+1}$ such that $f=\pi \circ \tilde{f}$.

Let $D$ be the Levi-Civita connection of $\mathbb{C}^{m+1}$ and $\tilde{B}$ the second fundamental form of $M^m$ in $S^{2m+1}$.
Then we have $\pi_*\tilde{B}=B$ and 
\begin{equation}\label{Leg}
D_XY=\nabla_XY+\tilde{B}(X,Y)-\langle X,Y\rangle \tilde{f}.
\end{equation}

We denote as $\nabla_{e_i}e_j=\sum^m_{l=1}\omega_{j}^l(e_i)e_l$ $(i,j=1,\cdots,m)$.
 Then we have

\vspace{5pt}

\begin{lem}[\cite{byc4},\cite{ts2}]\label{lem2}
Let $M^m$ be an $m-$dimensional Lagrangian $H$-umbilical submanifold in a complex space form. 
 For an orthonormal frame field $\{e_i\}_{i=1}^m$, we have
 \begin{align}
 &e_j\lambda =(2\mu -\lambda)\omega^1_{j}(e_1),\ \ \ j>1,\label{4.2}\\
 &e_1\mu=(\lambda-2\mu)\omega^l_1(e_l),\ \ \ \text{for all}\  l=2,\cdots m,\label{4.3}\\
 &(\lambda-2\mu)\omega^i_1(e_j)=0,\ \ \ i\neq j>1,\label{4.4}\\
 &e_j\mu=0,\ \ \ j>1,\label{4.5}\\
 &\mu\omega^j_1(e_1)=0\label{4.6}\\
 &\mu \omega^2_1(e_2)=\cdots =\mu \omega^m_1(e_m),\label{4.7}\\
 &\mu \omega^i_1(e_j)=0,\ \ \ i\neq j>1.\label{4.8}
 \end{align}
\end{lem}

\vspace{5pt}

\begin{proof}
By $(\nabla^{\perp}_{e_j}B)(e_1,e_1)=(\nabla^{\perp}_{e_1}B)(e_j,e_1)$ and $(\ref{4.1})$, we obtain $(\ref{4.2})-(\ref{4.4})$.\\
By $(\nabla^{\perp}_{e_1}B)(e_j,e_j)=(\nabla^{\perp}_{e_j}B)(e_1,e_j)$ and $(\ref{4.1})$, we obtain $(\ref{4.5})$ and $(\ref{4.6})$.\\
By $(\nabla^{\perp}_{e_i}B)(e_j,e_j)=(\nabla^{\perp}_{e_j}B)(e_i,e_i)$, $(i \ne j>1)$, and $(\ref{4.1})$, we obtain $(\ref{4.7})$ and $(\ref{4.8})$.
\end{proof}

Then, B. Y. Chen also showed the following:

\vspace{5pt}

\begin{thm}[\cite{byc4}]\label{cby}
Let $\p:M^m\rightarrow {\mathbb CP}^m(4)$ be a Lagrangian $H$-umbilical immersion with $m\geq 3$. If $M^m$
 contains no open subsets of constant sectional curvature $\geq 1$, then, up to rigid motions of $\mathbb{CP}^m(4)$, 
$\p$ is congruent to the immersion given the following:
 
 Let $z=(z_1,z_2):I\rightarrow S^3(1)\subset \mathbb{C}^2$ be an unit speed Legendre curve satisfying
 \begin{equation}\label{z}
 z''(x)=i\lambda(x)z'(x)-z(x),
 \end{equation}
 for some function $\lambda(x),$ where $I$ is an open interval of $\mathbb{R}$ or a circle.
  Assume that $|z_2(x)|$ is a positive function. Let $I\times _{|z_2(x)|}S^{m-1}(1)$ be a Riemannian manifold endowed with the metric
   $g=dx^2+|z_2|^2 g_0$, where $g_0$ is the standard metric on $S^{m-1}(1)$. Then $\p:I\times _{|Z_2(x)|}S^{m-1}(1)\rightarrow \mathbb{CP}^m(4)$,
   given by
\begin{equation}\label{4.10}
\p(x,y_1,\cdots ,y_m)=\pi (z_1(x), z_2(x)y_1, \cdots , z_2(x)y_m),
\end{equation}
with ${y_1}^2+\cdots +{y_m}^2=1$, defines a Lagrangian $H$-umbilical immersion with respect to an orthonormal local frame field 
 $e_1, \cdots ,e_m$ with $e_1=\frac{\partial}{\partial x}$, where $\mu =\frac{1}{|z_2(x)|^2}Re(i z_2\bar z_2'),$
 $\pi :S^{2m+1}(1)\rightarrow \mathbb{CP}^m(4)$ is the Hopf fibration is defined by 
 $$z\rightarrow z\mathbb{C}.$$
\end{thm}

\vspace{5pt}

Due to this theorem, we shall classify biharmonic PNMC Lagrangian $H$-umbilical submanifolds in complex space forms.
We first give necessary and sufficient conditions for Lagrangian $H$-umbilical submanifolds in complex space forms to be biharmonic.

\begin{lem}\label{nsbih}
Let $M^m$ be a Lagrangian $H$-umbilical submanifold in a complex space form $(N^m(4\epsilon), J, \langle \cdot , \cdot \rangle)$. Then, $M^m$ is biharmonic if and only if 
\begin{align}\label{1}
2\,\lambda \,(e_1 a)+a\,(e_1 \lambda) +\lambda \,a \sum^m_{l=2}\omega^l_1(e_l)=0,
\end{align}
\begin{align}\label{2}
2\,\mu (e_j a)+a\,\lambda \,\omega^j_1(e_1)=0,\ \ j>1,
\end{align}
\begin{align}\label{5.3}
&-\sum^m_{i=1} e_i (e_i a) +a\sum^m_{i,j=1}\omega^j_1(e_i)^2
+\sum^m_{i,j=1}\left(e_j a\right)\omega^j_i(e_i)\notag\\
&\hspace{60pt}+a\left\{\lambda^2+(m-1)\mu^2-\epsilon (m+3)\right\}=0,
\end{align}
\begin{align}\label{5.4}
&-2\sum^m_{i=1}(e_i a)\omega^j_1(e_i)
-a \sum^m_{i=1}e_i\left(\omega^j_1(e_i)\right)\notag\\
&\hspace{20pt}-a\sum^m_{i,l=1}\omega^l_1(e_i)\omega^j_l(e_i)+a\sum^m_{i,l=1}\omega^l_i(e_i)\omega^j_1(e_l)=0.
\end{align}
\end{lem}

\vspace{5pt}

\begin{proof} 
We shall calculate the tangential part $(\ref{tan})$. Using Lemma $\ref{lem2}$, we have
\begin{equation}\label{tan1}
\begin{aligned}
{\rm trace}_g\left(A_{\nabla^{\perp}_{\cdot}}(\cdot)\right)
=&\sum^m_{i=1}A_{\nabla^{\perp}_{e_i}a J e_1}e_i\\
=&\sum^m_{i=1}(e_i a)A_{J e_1}e_i+a\sum^m_{i,l=1}\omega^l_1(e_i)A_{Je_l}e_i\\
=&\lambda \sum^m_{i=1}(e_i a)e_1+\mu\sum^m_{i=2}(e_i a)e_i\\
&+a\mu\sum^m_{l=2}\omega^l_{1}(e_1)e_l
+a\mu \sum^m_{l=2}\omega^l_1(e_l)e_1\\
=&\lambda \sum^m_{i=1}(e_i a)e_1+\mu\sum^m_{i=2}(e_i a) e_i
+a\mu \sum^m_{l=2}\omega^l_1(e_l)e_1,
\end{aligned}
\end{equation}
and 
\begin{equation}\label{tan2}
\begin{aligned}
{\rm trace}_g\left(\nabla A_{\bf H}\right)
=&\sum^m_{i=1}\nabla_{e_i}\left( A_{{\bf H}e_i}\right)-\sum^m_{i=1}A_{\bf H}\left( \nabla_{e_i}e_i\right)\\
=&\sum^m_{i=1}\nabla_{e_i}A_{a J e_1} e_i
-a\sum^m_{i,l=1}A_{Je_1}\left(\omega^l_{i}(e_i)e_l\right)\\
=&\sum^m_{i=1}\left\{(e_i a)A_{J e_1}e_i+a\nabla_{e_i}A_{J e_1}e_i\right\}
-a\sum^m_{i,l=1}\omega^l_i(e_i)A_{Je_1}e_l\\
=&\lambda (e_1 a)e_1+a(e_1\lambda) e_1+a\lambda \sum^m_{l=1}\omega^l_1(e_1)e_l\\
&+\sum^m_{i=2}\left\{\mu (e_i a)e_i+a(e_i \mu )e_i+ a\mu \sum^m_{l=1}\omega^l_i(e_i)e_l\right\}\\
&-a\lambda\sum^m_{l=1}\omega^1_l(e_l)e_1-a\mu\sum^m_{l=1}\sum^m_{i=2}\omega^i_l(e_l)e_i.
\end{aligned}
\end{equation}
By $(\ref{tan1})$ and $(\ref{tan2})$, we have obtain
\begin{equation}
\begin{aligned}
{\rm trace}_g\left(\nabla A_{\bf H}\right)+{\rm trace}_g\left(A_{\nabla^{\perp}_{\cdot}}(\cdot)\right)
=&\left\{2\lambda(e_1 a)+a(e_1 \lambda)+a\lambda\sum^m_{l=2}\omega^l_1(e_l)\right\}e_1\\
&+\sum^m_{j=2}\left\{2\mu (e_j a)+a\lambda \omega^j_1(e_1)\right\}e_j,
\end{aligned}
\end{equation}
 which yields $(\ref{1})$ and $(\ref{2})$.
 
Next, we shall calculate the normal part $(\ref{nor})$. 
 Using Lemma \ref{lem2}, we have
\begin{equation}\label{nor1}
\begin{aligned}
\lapla^{\perp}{\bf H}
=&-\sum^m_{i=1}\nabla^{\perp}_{e_i}\nabla^{\perp}_{e_i}(aJe_1)+\sum^m_{i=1}\nabla^{\perp}_{\nabla_{e_i}e_i}(aJe_1)\\
=&-\sum^m_{i=1}(e_ie_i a)Je_1
-\sum^m_{i,j=1}\left\{2(e_i a)\omega^j_1(e_i)Je_j+ae_i\left(\omega^j_1(e_j)Je_l\right)\right\}\\
&-a\sum^m_{i,j,l=1}\omega^j_1(e_i)\omega^l_j(e_i)Je_l
+\sum^m_{i,j=1}\omega^j_1(e_i)(e_j a)Je_1\\
&+a\sum^m_{i,j,l=1}\omega^j_i(e_i)\omega^l_1(e_j)Je_l,
\end{aligned}
\end{equation}
and 
\begin{align}\label{nor2}
{\rm trace}_g B\left(A_{\bf H}(\cdot),\cdot\right)
=a\left\{\lambda^2+(m-1)\mu^2\right\}Je_1.
\end{align}
By $(\ref{nor1})$ and $(\ref{nor2})$, we obtain
\begin{align*}
&\lapla^{\perp}{\bf H}+{\rm trace}_g B\left(A_{\bf H}(\cdot),\cdot\right)-(m+3)\epsilon {\bf H}\\
&=\Big\{-\sum^m_{i=1} e_i e_i a +a\sum^m_{i,j=1}\omega^j_1(e_i)^2
+\sum^m_{i,j=1}(e_j a)\omega^j_i(e_i)\\
&\hspace{20pt}+a\left\{\lambda^2+(m-1)\mu^2-\epsilon (m+3)\right\}\Big\}Je_1\\
&+\sum^m_{j=2}\Big\{-2\sum^m_{i=1}(e_i a)\omega^j_1(e_i)
-a \sum^m_{i=1}e_i\left(\omega^j_1(e_i)\right)\\
&\hspace{30pt}-a\sum^m_{i,l=1}\omega^l_1(e_i)\omega^j_l(e_i)
+a\sum^m_{i,l=1}\omega^l_i(e_i)\omega^j_1(e_l)\Big\}Je_j,
\end{align*}
which yields $(\ref{5.3})$ and $(\ref{5.4})$.
\end{proof}

\vspace{5pt}


From Lemma \ref{nsbih}, we obtain the following.
\vspace{5pt}

\begin{lem}\label{main}
Let $M^m$ be a biharmonic Lagrangian $H$-umbilical submanifold in a complex space form $(N^m(4\epsilon), J, \langle \cdot , \cdot \rangle)$. Then, we have the
following equations:
\begin{equation}\label{1'}
2\lambda\left(e_1 a\right)+a\left(e_1 \lambda\right)+a\lambda(m-1)k=0,
\end{equation}
\begin{equation}\label{2'}
e_j a=0,\ \ \ j>1,
\end{equation}
\begin{equation}\label{5.3'}
-e_1(e_1 a)+a(m-1)k^2-\left(e_1 a\right)(m-1)k+a\left\{\lambda^2+(m-1)\mu^2-\epsilon (m+3)\right\}=0,
\end{equation}
\begin{equation}\label{5.4'}
e_j k=0,\ \ \ j>1,
\end{equation}
where, $k=\omega^2_1(e_2)=\cdots =\omega^m_1(e_m)$.
\end{lem}

\vspace{5pt}

\begin{proof}
First, we shall show $\mu \neq0$.
Assume that $\mu=0,$ then $a=\frac{1}{m}\lambda\neq 0$.
By Lemma {\ref{lem2}}, we have
\begin{equation}\label{3}
\omega^i_1(e_j)=0,\ \ j=2,\cdots, m.
\end{equation}
From $(\ref{1})$, $e_1 a=0$.
Then, from $(\ref{2})$, we obtain 
\begin{equation}\label{omega1}
\omega^j_1(e_1)=0,\ \ j=1,\cdots,m.
\end{equation}
Combining  $(\ref{3})$ and $(\ref{omega1})$, we have
\begin{equation}\label{omega2}
\omega^i_1(e_j)=0,\ \  i,j=1,\cdots,m .
\end{equation}
It follows that $\langle R(e_1,e_i)e_i, e_1\rangle=0.$ Thus, by $(\ref{GCeq})$, we have $\epsilon=0.$
By $(\ref{4.2})$, we have $e_j a=0,\ (j>1).$
From this and $(\ref{5.3})$, we obtain $a=0$, which contradicts the assumption.

 Thus, we only have to consider the case of $\mu \neq 0$. Then, we have
 \begin{align}
 \omega^i_1(e_j)=0,\ \ i\neq j,\label{5.12}\\
 \omega^2_1(e_2)=\cdots =\omega^m_1(e_m)\label{5.13}.
 \end{align} 
We put $k=\omega^2_1(e_2)=\cdots =\omega^m_1(e_m)$.

By $(\ref{5.13})$, we can denote that the equation $(\ref{1})$ is $(\ref{1'})$.
Putting $(\ref{5.12})$ into $(\ref{2})$, we obtain $(\ref{2'})$.
From $(\ref{2'})$ and $(\ref{5.3})$, we have $(\ref{5.3'})$.
Putting $(\ref{2'})$ into $(\ref{5.4})$, we have
$$-a \sum^m_{i=1}e_i\left(\omega^j_1(e_i)\right)
-a\sum^m_{i,l=1}\omega^l_1(e_i)\omega^j_l(e_i)
+a\sum^m_{i,l=1}\omega^l_i(e_i)\omega^j_1(e_l)=0.
$$
Thus, form $(\ref{5.12})$ we obtain $(\ref{5.4'})$.
\end{proof}

\vspace{10pt}

Using these results, we shall classify all the biharmonic PNMC Lagrangian $H$-umbilical submanifolds in complex space forms.






\vspace{10pt}


 \begin{thm}\label{blhsc}
 Let $\p:M^m\rightarrow (N^m(4\epsilon), J,\langle \cdot , \cdot \rangle)$ be a Lagrangian $H$-umbilical isometric immersion into a complex space form which has PNMC.
  Then, $\p$ is biharmonic if and only if $\epsilon =1$ and $\p(M)$ is congruent to 
 \begin{align}\label{5.2}
  \pi \left( \sqrt{\frac{\mu^2}{\mu^2+1}}e^{-\frac{i}{\mu}x}, \sqrt{\frac{1}{\mu^2+1}}e^{i\mu x}y_1, \cdots ,\sqrt{\frac{1}{\mu^2+1}}e^{i\mu x}y_m \right) \subset \mathbb{CP}^m(4), 
   \end{align}
  where $x$, $y_1,\cdots, y_m$ run through real numbers satisfying ${y_1}^2+\cdots +{y_m}^2=1$.
  Here, $\mu=\pm\sqrt{\frac{m+5\pm \sqrt{m^2+6m+25}}{2m}}$.
 \end{thm}
 
 \vspace{5pt}
 
 \begin{proof}
  By the assumption 
 $$\nabla^{\perp} \left(\frac{\bf H}{|\bf H|}\right)=\nabla^{\perp}\left(\frac{aJe_1}{|a|}\right)=0,$$
 and $a\neq 0$, we have
 \begin{align}
J(\nabla e_1)=\nabla^{\perp} Je_1 = 0.
 \end{align}
 Thus, we obtain
 \begin{equation}
 0=\nabla _{e_i} e_1=\sum^m_{l=1}\omega^l_1(e_i)e_l,\ \ \ \ (i=1,\cdots, m).
 \end{equation}
 Therefore, we have
 \begin{equation}\label{o0}
 \omega^l_1(e_i)=0,\ \ \ \ (i,l=1,\cdots ,m).
 \end{equation}
 Especially, we have
 \begin{equation}\label{k0}
 k=\omega^2_1(e_2)=\cdots=\omega^m_1(e_m)=0.
  \end{equation}

 By $(\ref{4.3})$, we obtain
  \begin{align}\label{5.15}
  e_1\mu=0.
 \end{align}
Thus, $\mu$ is constant. Since $\langle R(e_i,e_1)e_1,e_i \rangle=0,$ we have
\begin{align}\label{5.16}
\mu^2-\lambda\mu=\epsilon.
\end{align}
Thus, $\lambda$ is constant.
Therefore, $a=\frac{\lambda+(m-1)\mu}{m}$ is a non-zero constant, and the equation $(\ref{5.3'})$ is
\begin{equation}\label{5.3''}
\lambda^2+(m-1)\mu^2-\epsilon (m+3)=0.
\end{equation}
The equation $(\ref{5.3''})$ implies that $\epsilon >0$.
Using $(\ref{5.16})$ and $(\ref{5.3''})$, we obtain
\begin{align}
\lambda=&\frac{\mu^2-1}{\mu},\label{lambda}\\
\mu=&\pm \sqrt{\frac{m+5\pm \sqrt{m^2+6m+25}}{2m}}\label{mu}.
\end{align}

When $m=2$,  let $\tilde{f}$ be a Legendrian immersion such that $f=\pi \circ \tilde{f}$.
 Using the results of Hiepko (cf. \cite{sh1},~\cite{ts3}), there exists a local coordinate system $(x,y)$ 
 on $M$ such that the metric is given by $g=dx^2+G(x)dy^2$, for some function $G(x)$ with $e_1=\frac{\partial}{\partial x}$, $e_2=G^{-1}\frac{\partial}{\partial y}$.
  By $(\ref{o0})$, we have $G=1$, namely, $g=dx^2+dy^2$.
  From this and $(\ref{Leg})$, we obtain the following differential equations.
 \begin{align}
 \tilde{f}_{xx}=&\sqrt{-1}\lambda\tilde{f}_x-\tilde{f},\label{f1}\\
 \tilde{f}_{xy}=&\sqrt{-1}\mu\tilde{f}_y,\label{f2}\\
 \tilde{f}_{yy}=&\sqrt{-1}\mu\tilde{f}_x-\tilde{f}\label{f3}.
  \end{align}
  By $(\ref{f1})$, $(\ref{f2})$ and $(\ref{f3})$, we obtain

\begin{equation}
\begin{aligned}
\tilde{f}(x,y)=&\left( c_2 \sin{\sqrt{\mu^2+1}y}+c_3\cos{\sqrt{\mu^2+1}y}\right)e^{\sqrt{-1}\mu x}\\
&+c_1\sqrt{-1}\frac{\mu}{\mu^2+1}e^{\sqrt{-1}(\lambda-\mu)x},
\end{aligned}
\end{equation}
for some constant vectors $c_1$, $c_2$ and $c_3$ in $\mathbb{C}^3$. 
Since $\tilde{f}$ is a Legendrian immersion,
 $|c_1|=\sqrt{\mu^2+1}$, $|c_2|=|c_3|\frac{1}{\sqrt{\mu^2+1}}$ and 
 $\langle c_i,c_j\rangle =\langle c_i, \sqrt{-1} c_j \rangle =0,~(i\neq j)$.
 Thus, we can choose $c_1=(-\sqrt{-1}\sqrt{\mu^2+1},0,0)$,
  $c_2=(0,0,\frac{1}{\sqrt{\mu^2+1}})$ and 
  $c_3=(0,\frac{1}{\sqrt{\mu^2+1}},0)$.
 Therefore, we have $(\ref{5.2})$.

 When $m>2$,  
 From $(\ref{GCeq})$, we have that $M$ contains no open subsets of constant sectional curvature bigger than $1$.
 By Theorem $\ref{cby}$, $\p$ is congruent to $(\ref{4.10})$.
  From $(\ref{z})$ and $(\ref{lambda})$, we have $(z_1(x),z_2(x))=\left(\sqrt\frac{\mu^2}{\mu^2+1}e^{-\frac{i}{\mu}x},\sqrt \frac{1}{\mu^2+1}e^{i\mu x}\right)$.

Conversely, by a direct computation, it turns out that the immersion $(\ref{5.2})$ is a biharmonic PNMC Lagrangian immersion.
 \end{proof}
 
 \qquad\\
 \qquad\\
 


\quad\\
\quad\\

\noindent
Shun~MAETA~\\
Division of Mathematics, Graduate School of Information Sciences,\\
 Tohoku University, Sendai 980-8579, Japan\\
e-mail:~maeta@ims.is.tohoku.ac.jp\\

\noindent
Hajime~URAKAWA\\
Institute for International Education Tohoku University,
Sendai 980-8576, Japan\\
e-mail:urakawa@math.is.tohoku.ac.jp

\end{document}